\documentclass[11pt]{amsart}

\usepackage{url}
\usepackage{amsmath}
\usepackage{amsfonts}
\usepackage{enumitem}

\newcommand{\al}{\alpha}
\newcommand{\PF}{{\rm PF}}
\newcommand{\F}{{\rm F}}
\newcommand{\RF}{{\rm RF}}

\newtheorem{theorem}{Theorem}
\newtheorem{lemma}{Lemma}
\newtheorem{proposition}{Proposition}
\newtheorem{corollary}{Corollary}

\newtheorem{example}{Example}

\newtheorem{remark}{Remark}

\begin{document}

\title{Almost Symmetric Numerical Semigroups with Odd Generators}

	\author{Francesco Strazzanti}
	\address{Francesco Strazzanti - Institut de Matem\`{a}tica - Universitat de Barcelona - Gran Via de les Corts Catalanes 585 - 08007 Barcelona - Spain}
	\email{francesco.strazzanti@gmail.com}

\author{Kei-ichi Watanabe}
	\address{Kei-ichi Watanabe - Department of Mathematics - College of Humanities and Sciences - Nihon University - Setagaya-ku - Tokyo 156-8550 - Japan}
	\email{watanabe@math.chs.nihon-u.ac.jp}

\begin{abstract}
We study almost symmetric semigroups generated by odd integers. If the embedding dimension is four, we characterize when a symmetric semigroup that is not complete intersection or a pseudo-symmetric semigroup is generated by odd integers. Moreover, we give a way to construct all the almost symmetric semigroups with embedding dimension four and type three generated by odd elements. In this case we also prove that all the pseudo-Frobenius numbers are multiple of one of them and this gives many consequences on the semigroup and its defining ideal.
\end{abstract}

\maketitle

\section{Introduction}

Numerical semigroups have been extensively studied in the last decades for several reasons, since they appears in many areas of mathematics like commutative algebra, algebraic geometry, number theory, factorization theory, combinatorics or coding theory. For instance, the connection with commutative algebra has greatly influenced the theory of numerical semigroups and it is not a coincidence that many invariants of numerical semigroups have the same name of well-known invariants in commutative algebra. One of the main results that constructed a bridge between these two areas is the celebrated theorem of Kunz \cite{K} that establishes the equivalence between Gorenstein rings and symmetric numerical semigroups. More precisely, if $R$ is a one-dimensional analytically irreducible and residually rational noetherian local ring, then it is Gorenstein if and only if the associated value-semigroup (that is a numerical semigroup) is symmetric.

An important notion related to the symmetry of a numerical semigroup is given by the pseudo-symmetric property. The rings that correspond to the pseudo-symmetric semigroups are called Kunz rings by many authors and there is an extensive literature about them. See for instance the monograph \cite{BDF} that also provides a dictionary between commutative algebra and numerical semigroup theory.

In 1997 Barucci and Fr\"oberg \cite{BF} introduced the notion of almost symmetric numerical semigroup that generalizes both symmetric and pseudo-symmetric ones. Similarly, in the same paper they introduced almost Gorenstein ring as the correspondent notion in commutative algebra; of course, it generalizes Gorenstein and Kunz rings. The last definition is given in the one-dimensional analytically unramified local case, but recently it was extended in the one-dimensional and higher dimension local case as well as in the graded context, see \cite{GMP,GTT}.

On the other hand almost symmetric semigroups have been studied by many authors from several points of view. They are also one of the main tools used in \cite{OST} to construct one-dimensional Gorenstein local rings with decreasing Hilbert functions in some level, giving an answer to a commutative algebra problem known as Rossi Problem.
There are also many generalizations of the almost symmetric semigroups in literature, see \cite{CGKM,DS,GK,HHS}. 

The purpose of this paper is to study the almost symmetric semigroups generated by odd integers, in particular when the embedding dimension is four. In this case Moscariello \cite{M} proved that the type of the semigroup is at most three confirming a conjecture of T. Numata. This means that we can divide the almost symmetric semigroups with embedding dimension four in three classes: symmetric, pseudo-symmetric and having type three. 

If $S=\langle n_1, \dots, n_e \rangle$ is a numerical semigroup, we say that $k[S]:=k[t^s \mid s \in S]$ is the numerical semigroup ring associated to $S$, where $k$ is a field and $t$ is an indeterminate. It is possible to present this ring as a quotient of a polynomial ring $k[S] \cong k[x_1, \dots, x_e]/I_S$ and $I_S$ is called the defining ideal of $S$. We set $\deg(x_i)=n_i$ for every $i=1, \dots, e$, thus $I_S$ is homogeneous.

Assume now that $S$ has embedding dimension four. In the case of symmetric and pseudo-symmetric numerical semigroups the defining ideal is known by Bresinsky \cite{B} and Komeda \cite{Ko}. The case with type three has been recently studied in \cite{E,HW}, where the defining ideal is found using the notion of $\RF$-matrix, introduced in \cite{M}.

We focus on the case where all the generators of $S$ are odd. In particular, if $S$ is symmetric but not complete intersection we characterize when this happens in terms of some numbers related to the defining ideal of $S$. Moreover, in the pseudo-symmetric case we connect this property to the rows of a suitable \RF-matrix associated to $S$. 
If $S$ is almost symmetric with type three, we prove that the set of the pseudo-Frobenius numbers of $S$ is $\PF(S)=\{f,2f,3f\}$ for some integer $f$. This lead to the description of the generators of both $S$ and $I_S$ as well as the minimal free resolution of $k[x_1, \dots, x_4]/I_S$ in terms of the numbers $\al_i=\min\{\al \mid \al n_i \in \langle n_1, \dots, \widehat{n_i}, \dots, n_4 \rangle\}$ for $i=1, \dots, 4$, where $n_1$, $n_2$, $n_3$ and $n_4$ are the minimal generators of $S$.  This allows us to construct all such semigroups and gives examples of numerical semigroups in which $\PF(S)$ has this particular shape, which have been studied in \cite{GKMT}.

The structure of the paper is the following. In Section \ref{basic concepts} we fix the notation and recall some useful definitions and results. In Section \ref{section symmetric} we characterize when the generators of a symmetric numerical semigroup with embedding dimension four are all odd. In Section \ref{section pseudo-symmetric} we do the same in the pseudo-symmetric case. In the last section we consider the case of almost symmetric semigroups with embedding dimension four and type three. Here we prove Theorem \ref{type 3} which gives the pseudo-Frobenius numbers and that allows to get Corollary \ref{HW}, where the generators of $S$ and $I_S$ as well as the minimal free resolution of $k[x_1, \dots, x_4]/I_S$ are described. Moreover, in Theorem \ref{construction} we give a way to construct all the almost symmetric semigroups with embedding dimension four and type three.

Several computations of the paper are performed by using the GAP system \cite{GAP} and, in particular, the NumericalSgps package \cite{DGM}.

\section{Basic Concepts} \label{basic concepts}

We denote by $\mathbb{N}$ the set of the natural numbers including $0$. A numerical semigroup $S$ is an additive submonoid of $\mathbb{N}$ such that $\mathbb{N} \setminus S$ is finite.
Every numerical semigroup has a finite system of generators, i.e. there exist some positive integers $n_1, n_2, \dots, n_s$ such that $S=\langle n_1, n_2,\dots, n_s \rangle:=\{\sum_{i=1}^s a_i n_i \mid a_i \in \mathbb{N} {\rm \ for \ } i=1, \dots, s\}$. Moreover, there exists a unique minimal system of generators $n_1, \dots, n_e$ of $S$ and the number $e$ is called embedding dimension of $S$. The finiteness of $\mathbb{N}\setminus S$ is equivalent to $\gcd(n_1, \dots, n_e)=1$. 
If $S=\langle n_1, \dots, n_e \rangle$ we denote by $\al_i$ the minimum integer such that $\al_i n_i=\sum_{j \neq i} a_j n_j$ for some non-negative integers $a_1, \dots, a_e$.

The maximum of $\mathbb{Z} \setminus S$ is known as the Frobenius number of $S$ and we denote it by $\F(S)$. We say that an integer $f \in \mathbb{Z} \setminus S$ is a pseudo-Frobenius number of $S$ if $f+s \in S$ for every $s \in S \setminus \{0\}$. We denote the set of the pseudo-Frobenius numbers by $\PF(S)$ and we refer to its cardinality $t(S)$ as the type of $S$. Clearly $\F(S)$ is always a pseudo-Frobenius number, thus $t(S) \geq 1$.

Consider the injective map $\varphi:S \rightarrow \mathbb{Z} \setminus S$ defined by $\varphi(s)=\F(S)-s$. If $\varphi$ is a bijection we say that $S$ is symmetric, whereas if the image of $\varphi$ is equal to $\mathbb{Z} \setminus S$ except for $\F(S)/2$ we say that $S$ is pseudo-symmetric. It is not difficult to see that $S$ is symmetric if and only if it has type $1$ and it is pseudo-symmetric if and only if $\PF(S)=\{\F(S)/2,\F(S)\}$. Moreover, setting $g(S)=|\mathbb{N}\setminus S|$, $S$ is symmetric (resp. pseudo-symmetric) if and only if $2g(S)=\F(S)+1$ (resp. $2g(S)=\F(S)+2$). We say that $S$ is almost symmetric if and only if $2g(S)=\F(S)+t(S)$. There exists a useful characterization of the almost symmetric property due to H. Nari \cite[Theorem 2.4]{N}: if $\PF(S)=\{f_1 < f_2 < \dots < f_t=\F(S)\}$, a numerical semigroup is almost symmetric if and only if $f_i+f_{t-i}=\F(S)$ for every $i=1, \dots, t-1$. 

If $f \in \PF(S)$, then $f+n_i \in S$ for every $i$ and, thus, there exist $\lambda_{i1}, \dots, \lambda_{ie} \in \mathbb{N}$ such that $f+n_i=\sum_{j=1}^e \lambda_{ij}n_j$. Since $f \notin S$, $\lambda_{ii}$ has to be equal to zero. For every $i,j = 1, \dots, e$, set $a_{ii}=-1$ and $a_{ij}=\lambda_{ij}$ if $i \neq j$. Following \cite{M} we say that the matrix $\RF(f)=(a_{ij})$ is a row-factorization matrix of $f$, briefly RF-matrix. Note that there could be several RF-matrices of $f$ and that $f=\sum_{j=1}^e a_{ij}n_j$ for every $i$. 
For instance, consider the numerical semigroup $S=\langle 8,10,11,13 \rangle$ that has embedding dimension four and is symmetric, because $\PF(S)=\{25\}$. The following are both \RF-matrices of $\F(S)=25$:

\begin{equation*}
\begin{pmatrix}
   -1    &    0    &    3    &    0    \\
    3    &   -1    &    1    &    0    \\
    2    &    2    &   -1    &    0    \\
    1    &    3    &    0    &   -1    \\
\end{pmatrix}, \ \ \ \  \ \ \ \ 
\begin{pmatrix}
   -1    &    2    &    0    &    1    \\
    0    &   -1    &    2    &    1    \\
    0    &    1    &   -1    &    2    \\
    2    &    0    &    2    &   -1    \\
\end{pmatrix}.
\end{equation*}

\vspace{10pt}

\section{Symmetric Semigroups} \label{section symmetric}

We start by studying the symmetric numerical semigroups with embedding dimension four. If the semigroup is not complete intersection, there is a theorem proved by Bresinsky \cite{B} that gives much information on the semigroup and its defining ideal. We state it following \cite[Theorem 3]{BFS}.  
By convention, if $i$ is an integer not included between $1$ and $4$, we set $a_i=a_j$ and $b_i=b_j$ with $i \equiv j \mod 4$ and $1 \leq j \leq 4$.

\begin{theorem}
Let $S$ be a numerical semigroup with $4$ minimal generators. Then, $S$ is symmetric and not complete intersection if and only if there are integers $a_i$ and $b_i$ with $i \in \{1, \dots, 4\}$ such that $0 < a_i < \al_{i+1}$ and $0 < b_i < \al_{i+2}$ for all $i$,
\begin{equation}\label{al_i}
\al_1=a_1+b_1, \ \al_2=a_2+b_2, \ \al_3=a_3+b_3, \ \al_4= a_4+b_4
\end{equation}
and
\begin{equation}\label{n_i}
\begin{split}
n_1=\al_2\al_3a_4+a_2b_3b_4, \ \ n_2=\al_3\al_4a_1+a_3b_4b_1,  \\
n_3=\al_1\al_4a_2+a_4b_1b_2, \ \ n_4=\al_1\al_2a_3+a_1b_2b_3.
\end{split}
\end{equation}
In this case $I_S=(f_1,f_2,f_3,f_4,f_5)$, where
\begin{align*}
&f_1=x_{1}^{\al_{1}}-x_{3}^{b_3}x_{4}^{a_4}, &&f_2=x_{2}^{\al_{2}}-x_{1}^{a_1}x_{4}^{b_4}, &&&f_3=x_{3}^{\al_{3}}-x_{1}^{b_1}x_{2}^{a_2}, \\
&f_4=x_{4}^{\al_{4}}-x_{2}^{b_2}x_{3}^{a_3}, &&f_5=x_{1}^{a_1}x_{3}^{a_3}-x_{2}^{a_2}x_{4}^{a_4}.  
\end{align*}
\end{theorem}

In this section we denote by $a_i$ and $b_i$ the integers that appear in the previous theorem.

\begin{theorem} \label{symmetric}
Let $S$ be a symmetric numerical semigroup minimally generated by $n_1, \ldots, n_4$ and assume that $S$ is not complete intersection. The following conditions are equivalent:
\begin{enumerate}[label=\arabic*)]
\item Every $n_i$ is odd.
\item One of the following holds:
\begin{enumerate}
	\item All the $\al_i$'s and the $a_i$'s are odd;
	\item There is exactly one index $i_0$ for which $\al_{i_0}$ is even. Moreover, $a_{i_0}$ and $a_{i_0-1}$ are odd, while the other $a_i$'s are even;
	\item All the $\al_i$'s are even and all the $a_i$'s are odd.
\end{enumerate}	
\end{enumerate}
\end{theorem}

\begin{proof}
Using the equalities (\ref{al_i}) and (\ref{n_i}) it is easy to see that the conditions {\it (a)}, {\it (b)} and {\it (c)} imply that all the generators are odd. 

Conversely, assume first that all the $\al_i$'s are odd and suppose by contradiction that $a_1$ is even. Since $n_2$ is odd, $a_3$ and $b_4$ are odd by (\ref{n_i}). Therefore, $a_4=\al_4-b_4$ 
and $b_3=\al_3-a_3$ are even. Then $n_1$ should be even by (\ref{n_i}). A contradiction!

Assume now that there is at least one $\al_i$ even. Without loss of generality, we can assume that $\al_1$ is even. Since $n_3$ and $n_4$ are odd, the equalities in (\ref{n_i}) imply that $a_4$, $b_1$, $b_2$, $a_1$ and $b_3$ are odd. 

Assume first that $a_2$ is even. Then, the first equality in (\ref{n_i}) implies that $\al_2$ and $\al_3$ are odd, so $a_3=\al_3-b_3$ is even. Moreover, since $n_2$ is odd, $\al_4$ is odd by (\ref{n_i}). Hence, we are in the case {\it (b)}.

Assume now that $a_2$ is odd. Then, $\al_2=a_2+b_2$ is even and it follows from the first equality in (\ref{n_i}) that $a_2$ and $b_4$ are odd. In particular, $\al_4=a_4+b_4$ is even and, again by (\ref{n_i}), $a_3$ is odd. Finally, we get that $\al_3=a_3+b_3$ is even and, then, we are in the case {\it (c)}.
\end{proof}

\begin{example} \rm
We note that all the cases of the previous theorem occur. All the following semigroups are symmetric but not complete intersections.\\
{\it (a)} Let $S= \langle 13,17,23,19 \rangle$. In this case 
\begin{align*}
&f_1=x_{1}^{5}-x_{3}^2x_{4}, &&f_2=x_{2}^{3}-x_{1}x_{4}^2, &&&f_3=x_{3}^{3}-x_{1}^{4}x_{2}, \\
&f_4=x_{4}^{3}-x_{2}^2x_{3}, &&f_5=x_{1}x_{3}-x_{2}x_{4},  
\end{align*}
in particular $\al_1=5$, $\al_2=\al_3=\al_4=3$ and $a_1=a_2=a_3=a_4=1$. \\[1mm]
{\it (b)} Let $S= \langle 13,17,33,25\rangle$. We have
\begin{align*}
&f_1=x_{1}^{7}-x_{3}^2x_{4}, &&f_2=x_{2}^{3}-x_{1}^{2}x_{4}, &&&f_3=x_{3}^{3}-x_{1}^{5}x_{2}^2, \\
&f_4=x_{4}^{2}-x_{2}x_{3}, &&f_5=x_{1}^2x_{3}-x_{2}^2x_{4},  
\end{align*}
therefore $\al_1=7$ and $\al_2=\al_3=3$ and $\al_4=2$. Moreover, $a_1=a_2=2$ and $a_3=a_4=1$. \\[1mm]
{\it (c)} Let $S= \langle 5,7,11,9 \rangle$. Then
\begin{align*}
&f_1=x_{1}^{4}-x_{3}x_{4}, &&f_2=x_{2}^{2}-x_{1}x_{3}, &&&f_3=x_{3}^{2}-x_{1}^{3}x_{2}, \\
&f_4=x_{4}^{2}-x_{2}x_{4}, &&f_5=x_{1}x_{3}-x_{2}x_{4}  
\end{align*}
and, thus, $\al_1=4$, $\al_2=\al_3=\al_4=2$ and $a_1=a_2=a_3=a_4=1$.
\end{example}

\begin{example} \rm
Unfortunately, in Theorem \ref{symmetric} it is not possible to characterize the parity of the generators by the parity of the $\al_i$'s, in fact we cannot eliminate the conditions on the $a_i$'s in {\it (a)}, {\it (b)} and {\it (c)}, as the following examples show. They are all symmetric, but not complete intersections.\\
{\it (a)} Consider the semigroup $S=\langle 90,91,97,93 \rangle$. Then
\begin{align*}
&f_1=x_{1}^{13}-x_{3}^{12}x_{4}^{2}, &&f_2=x_{2}^{3}-x_{1}^2x_{4}, &&&f_3=x_{3}^{13}-x_{1}^{13}x_{2}, \\
&f_4=x_{4}^{3}-x_{2}^2x_{3}, &&f_5=x_{1}^2x_{3}-x_{2}x_{4}^2,  
\end{align*}
and all the $\al_i$ are odd, but there is an even generator. In fact, $a_1$ and $a_4$ are even.  \\[1mm]
{\it (b)} Let $S=\langle 22,23,29,57\rangle$. We have
\begin{align*}
&f_1=x_{1}^{5}-x_{3}^3x_{4}, &&f_2=x_{2}^{2}-x_{1}x_{4}^4, &&&f_3=x_{3}^{5}-x_{1}^{4}x_{2}, \\
&f_4=x_{4}^{5}-x_{2}x_{3}^2, &&f_5=x_{1}x_{3}^{2}-x_{2}x_{4}.  
\end{align*}
In this case $\al_1=\al_3=\al_4=5$ and $\al_2$ is even. However a generator is even, since $a_4$ is odd. \\[1mm]
{\it (c)} Let $S=\langle 5,14,22,18 \rangle$. We have 
\begin{align*}
&f_1=x_{1}^{8}-x_{3}x_{4}, &&f_2=x_{2}^{2}-x_{1}^{2}x_{4}, &&&f_3=x_{3}^{2}-x_{1}^{6}x_{2}, \\
&f_4=x_{4}^{2}-x_{2}x_{3}, &&f_5=x_{1}^2x_{3}-x_{2}x_{4},  
\end{align*}
in particular all the $\al_i$'s are even, but three generators of $S$ are even. Note that $a_1$ is even.
\end{example}

\section{Pseudo-symmetric Semigroups} \label{section pseudo-symmetric}

Let $S=\langle n_1, n_2, n_3 \rangle$ be a non-symmetric numerical semigroup. In \cite{H} it is proved that the defining ideal of $S$ is generated by the maximal minors of the matrix 
\begin{equation}
\begin{pmatrix}
x_1^{\al}  &  x_2^{\beta}  &  x_3^{\gamma}  \\[2mm]
x_1^{\al'}  &  x_2^{\beta'}  &  x_3^{\gamma'}  \\
\end{pmatrix}
\end{equation}
for some positive integers $\al, \beta, \gamma, \al', \beta', \gamma'$. Moreover, by \cite[Corollary 3.3]{NNW}, $S$ is pseudo-symmetric if and only if $\al=\beta=\gamma=1$ or $\al'=\beta'=\gamma'=1$. 
Without loss of generality we assume that $\al'=\beta'=\gamma'=1$. 
In \cite[(2.1.1) pag. 69]{NNW} it is proved that 
$n_1=(\beta+1)\gamma +1$,
$n_2=(\gamma+1)\al +1$ and
$n_3=(\al+1)\beta +1$. Hence, it follows easily that $n_1$, $n_2$ and $n_3$ are odd if and only if either $\al$, $\beta$, $\gamma$ are odd or $\al$, $\beta$, $\gamma$ are even.

\vspace{2mm}

Now let $S=\langle n_1, n_2,n_3,n_4 \rangle$ be a pseudo-symmetric 4-generated numerical semigroup. By \cite[Theorem 4.3]{HW} $\F(S)/2$ has a unique \RF-matrix and for a suitable relabeling of the generators of $S$ we have
\begin{equation} \label{RF Pseudo symmetric}
\RF(\F(S)/2)=
\begin{pmatrix}
-1      & \al_2-1 &     0   &    0    \\
0       &    -1   & \al_3-1 &    0    \\
\al_1-1 &     0   &    -1   & \al_4-1 \\
\al_1-1 &     a   &     0   &   -1    \\
\end{pmatrix}
\end{equation}
for some non-negative integer $a$.

Given $f \in \PF(S)$ and $\RF(f) = (a_{ij})$, we say that the $i$-th row is even (resp. odd) if $\sum_{j=1}^4 a_{ij}$ is even (resp. odd).

\begin{proposition}
Assume that $S = \langle n_1,\ldots, n_4 \rangle$ is pseudo-symmetric and has embedding dimension $4$. Then, every $n_i$ is odd if and only if one of the following conditions hold:
\begin{enumerate}[label=\arabic*)]
\item $\F(S)/2$ is odd and every row of $\RF(\F(S)/2)$ is odd;
\item $\F(S)/2$ is even and every row of $\RF(\F(S)/2)$ is even.
\end{enumerate}
\end{proposition}

\begin{proof}
We can assume that the matrix (\ref{RF Pseudo symmetric}) is the $\RF$-matrix of $f:=\F(S)/2$.

Suppose first that every $n_i$ is odd and $f$ is odd. By the first row of (\ref{RF Pseudo symmetric}), $f=-n_1 + (\al_2-1)n_2$ and $(\al_2-1)$ has to be even, i.e. the first row is odd. The same argument works for the second row. The third row (and similarly the last one) $f=(\al_1-1)n_1-n_3+(\al_4-1)n_4$ yields immediately that $\al_1-1$ and $\al_4-1$ have the same parity and, thus, the row is odd. If $f$ is even we can use the same argument.

Assume now that Condition {\it 1)} holds. By the first two rows it follows that $\al_2-1$, $\al_3-1$ are even and, then, $n_1$ and $n_2$ are odd. Using the last row we have $\al_1-1+a$ even, thus $(\al_1-1)n_1+an_2$ is even and $n_4$ has to be odd. In the same way the third row implies that also $n_3$ is odd.

Finally, assume that Condition {\it 2} holds. By the first row we get that $\al_2-1$ is odd and then $n_1$ and $n_2$ have the same parity. By the second one follows that also $n_3$ has the same parity of $n_1$ and $n_2$. If they are even, the last row implies that 
$f=(\al_1-1)n_1+an_2-n_4$ and, since $f$ is even, also $n_4$ is even. This is a contradiction because $\gcd(n_1,n_2,n_3,n_4)=1$, therefore, $n_1$, $n_2$ and $n_3$ are odd. Moreover, in the last row we have $\al_1-1+a$ odd and, then, $n_4$ is odd.
\end{proof}

\begin{remark} \rm
Let $S=\langle n_1, \dots, n_4 \rangle$ be pseudo-symmetric and assume that $\F(S)/2$ and every $n_i$ are odd. The previous proposition implies that $\al_2$ and $\al_3$ are odd. Moreover, $\al_1$ and $\al_4$ have the same parity, but we cannot determine if they are even or odd. In fact, if $S=\langle 15,17,35,43 \rangle$ we have $\PF(S)=\{53,106\}$ and 
\begin{equation*}
\RF(53)=
\begin{pmatrix}
-1 &  4 &  0 &  0 \\
 0 & -1 &  2 &  0 \\
 3 &  0 & -1 &  1 \\
 3 &  3 &  0 & -1 \\
\end{pmatrix},
\end{equation*}
whereas if $T=\langle 57,61,123,163\rangle$, then  $\PF(T)=\{431,862\}$ and 
\begin{equation*}
\RF(431)=
\begin{pmatrix}
-1 &  8 &  0 &  0 \\
 0 & -1 &  4 &  0 \\
 4 &  0 & -1 &  2 \\
 4 &  6 &  0 & -1 \\
\end{pmatrix}.
\end{equation*}
\end{remark}

\vspace{10pt}

\section{Almost Symmetric Semigroups with Type Three}

Moscariello \cite{M} proved that an almost symmetric numerical semigroup with embedding dimension four has type at most three. Therefore, to complete the picture we need to study the almost symmetric semigroups with type three. We start with an easy lemma that is probably known, but we include it for the reader's convenience.

\begin{lemma}
Let $S=\langle n_1, \dots, n_r \rangle$ and assume that $\al_1=n_2$. Then $S=\langle n_1, n_2 \rangle$.
\end{lemma}

\begin{proof}
If $T=\langle n_1, n_2 \rangle$, then $T$ is symmetric and $\F(T)=n_1 n_2 - n_1 - n_2$. Suppose by contradiction that $n_3 \notin T$. Since $T$ is symmetric, $\F(T)-n_3 \in T$, i.e. $\F(T)-n_3=a n_1 + bn_2$ for some non-negative integers $a$ and $b$. Therefore, $(n_2-a-1)n_1=(b+1)n_2+n_3$ and, then, $\al_1 \leq n_2-a-1$ gives a contradiction.
\end{proof}

\begin{theorem} \label{type 3}
Let $S=\langle n_1, n_2, n_3, n_4\rangle$ be an almost symmetric numerical semigroup with type three and assume that all the generators are odd. Then, its pseudo-Frobenius numbers are $\PF(S)=\{f,2f,3f\}$ for some integer $f$ and, by a suitable change of order of $n_1, n_2, n_3, n_4$, there exists an \RF-matrix of $f$ and $2f$ of the following type:

\begin{equation*}
\begin{pmatrix}
   -1    & \al_2-1 &    0    &    0    \\
    0    &   -1    & \al_3-1 &    0    \\
    0    &    0    &   -1    & \al_4-1 \\
\al_1-1  &    0    &    0    &    -1   \\
\end{pmatrix}, \ \ \ \ 
\begin{pmatrix}
   -1    & \al_2-2 & \al_3-1 &    0    \\
    0    &   -1    & \al_3-2 & \al_4-1 \\
 \al_1-1 &    0    &   -1    & \al_4-2 \\
\al_1-2  & \al_2-1 &    0    &    -1   \\
\end{pmatrix}.
\end{equation*}
\end{theorem}

\begin{proof}
According to \cite[Theorems 3.6 and 4.8]{E}, we distinguish four cases that in \cite{E} are called UF1, UF2, nUF1 and nUF2. We will prove that only the last one is possible under our hypothesis. Let $f$ and $f'$ be the two pseudo-Frobenius numbers of $S$ different from its Frobenius number. \\[2mm]
{\bf Case UF1.} In this case, by a suitable change of order of $n_1$, $n_2$, $n_3$, $n_4$, there exist \RF-matrices of $f$ and $f'$ of the following type
\begin{equation*}
\begin{pmatrix}
   -1    & \al_2-1 &    0    &    0    \\
 \al_1-1 &   -1    &    0    &    0    \\
 \al_1-2 &    0    &   -1    &    1    \\
    0    & \al_2-2 &    1    &   -1    \\
\end{pmatrix}, \ \ \ \ 
\begin{pmatrix}
   -1    &    0    &    0    & \al_4-1 \\
    0    &   -1    &    1    & \al_4-2 \\
b_{41}-1 &  b_{32} &   -1    &    0    \\
  b_{41} &b_{32}-1 &    0    &   -1    \\
\end{pmatrix}
\end{equation*}
respectively and either $\al_2 = 2$ or $\al_4 = 2$.

If $\al_2 =2$, the first two lines of the first matrix give $n_2 = f + n_1$ and $f + n_2 = (\al_1 -1) n_1$. Hence, $2f = (\al_1-2) n_1$ and, since $n_1$ is odd, $\al_1$ has to be even; consequently $f=(\al_1/2 -1)n_1 \in S$ gives a contradiction. 

Assume now that $\al_4 = 2$. The second matrix implies that $n_3 = n_2 + f'$ and $f'+ n_3 = (b_{41}-1) n_1 + b_{32} n_2$. Then, 
\begin{equation} \label{UF1}
\begin{split}
2 n_3 &= n_3 + ( f' + n_3) - f' = ( n_2 + f' ) +  ((b_{41}-1) n_1 + b_{32} n_2) - f' \\
&= (b_{41} - 1) n_1 + (b_{32} + 1) n_2.
\end{split}
\end{equation}

Moreover, subtracting the first and the second rows of $\RF(f)$, we get $\al_1 n_1 = \al_2 n_2$. The previous lemma implies that $\al_1 < n_2$ and, then, $\gcd(n_1, n_2) = d > 1$. Since $d$ is odd, in light of the equality $(\ref{UF1})$ also $n_3$ is a multiple of $d$.   
Furthermore, $\al_4 = 2$ means that $2n_4 = \sum_{ j = 1}^3 \al_{4j} n_j$ and, thus, also $n_4$ is a multiple of $d$; a contradiction. \\[2mm]
{\bf Case UF2}. By a suitable change of order of $n_1$, $n_2$, $n_3$, $n_4$, there exist \RF-matrices of $f$ and $f'$ of the following type: 
\begin{equation*}
\begin{pmatrix}
   -1    & \al_2-1 &    0    &    0    \\
 a_{21}  &   -1    & \al_3-2 &    0    \\
a_{21}-1 &    0    &   -1    &    1    \\
    0    & \al_2-2 & \al_3-1 &   -1    \\
\end{pmatrix}, \ \ \ \ 
\begin{pmatrix}
      -1      &    0    &    0    & \al_4-1 \\
       0      &   -1    & \al_3-1 & \al_4-2 \\
a_{21}+b_{41} &    0    &   -1    &    0    \\
     b_{41}   & \al_2-1 &    0    &   -1     \\
\end{pmatrix}.
\end{equation*}
respectively and either $\al_2 = 2$ or $\al_4 = 2$.

Assume first that $\al_2=2$. By subtracting the first and the last row in the first matrix we get
\begin{equation*}
n_2+n_4=n_1+(\al_3-1)n_3.
\end{equation*}
This implies that $\al_3$ is even. Therefore, by adding the first two rows of the first matrix we have
\begin{equation*}
2f=(a_{21}-1)n_1 + (\al_3-2)n_3
\end{equation*}
and $a_{21}-1$ has to be even. It follows that $f$ is in the semigroup, that is a contradiction.

Assume now that $\al_4=2$. In this case $f'=n_4-n_1$ is even, then $f=(\al_2-1)n_2-n_1$ is odd and, thus, $\al_2-1$ is even. 
By adding the first and the last row of the second matrix we get
\begin{equation*}
2f'=(b_{41}-1)n_1+(\al_2-1)n_2.
\end{equation*}
Again $b_{41}-1$ has to be even and, then, $f'$ is in the semigroup.\\[2mm]
{\bf Case nUF1.} By a suitable change of order of $n_1$, $n_2$, $n_3$, $n_4$, there exist \RF-matrices of $f$ and $f'$ of the following type:
\begin{equation*}
\begin{pmatrix}
   -1    &    0    &    0    & \al_4-1 \\
    0    &   -1    &    1    & \al_4-2 \\
    0    & \al_2-1 &   -1    &     0   \\
    1    & \al_2-2 &    0    &    -1   \\
\end{pmatrix}, \ \ \ \ 
\begin{pmatrix}
   -1    &    1    & \al_3-2 &    0    \\
 \al_1-1 &   -1    &    0    &    0    \\
 \al_1-2 &    0    &   -1    &    1    \\
    0    &    0    & \al_3-1 &   -1    \\
\end{pmatrix}
\end{equation*}
respectively. Subtracting the first two rows of the first matrix we get $n_1+n_3=n_2+n_4$, whereas subtracting the second and the third row we get $\al_2 n_2= 2n_3+(\al_4-2)n_4$ and, then, $\al_2$ and $\al_4$ have the same parity.
In the same way, by subtracting the first two rows of the second matrix, we get that $\al_1$ and $\al_3$ have the same parity.
By adding the first and the third row of the first matrix and using $n_1+n_3=n_2+n_4$ we have
\begin{equation*}
2f=-n_1+(\al_2-1)n_2 -n_3+(\al_4-1)n_4=(\al_2-2)n_2+(\al_4-2)n_4.
\end{equation*}
Since $f$ is not in the semigroup, this implies that $\al_2$ and $\al_4$ are odd and, thus, $f$ is odd by the first row.

If we do the same in the second matrix (with the second and the last row) we conclude that also $f'$ is odd, that is a contradiction because $f+f'$ equals the Frobenius number that is odd. \\[2mm]
{\bf Case nUF2.} By a suitable change of order of $n_1$, $n_2$, $n_3$, $n_4$, there exist \RF-matrices of $f$ and $f'$ of the following type:
\begin{equation*}
\begin{pmatrix}
   -1    & \al_2-1 &    0    &    0    \\
    0    &   -1    & \al_3-1 &    0    \\
    0    &    0    &   -1    & \al_4-1 \\
\al_1-1  &    0    &    0    &    -1   \\
\end{pmatrix}, \ \ \ \ 
\begin{pmatrix}
   -1    & \al_2-2 & \al_3-1 &    0    \\
    0    &   -1    & \al_3-2 & \al_4-1 \\
 \al_1-1 &    0    &   -1    & \al_4-2 \\
 \al_1-2 & \al_2-1 &    0    &    -1   \\
\end{pmatrix}
\end{equation*}
respectively. Since the sum of the first two rows of the first matrix is equal to the first row of the second matrix, it follows that $f'=2f$. Hence, it is enough to recall that $\F(S)=f+f'=3f$ by Nari's Theorem \cite[Theorem 2.4]{N}.
\end{proof}

\begin{remark} \label{al_i odd} \rm
Let $S=\langle n_1, n_2, n_3, n_4 \rangle$ be almost symmetric with  type three and assume that all the generators are odd.
By Theorem \ref{type 3} the Frobenius number is equal to $3f$ and it is odd, so $f$ is odd. Moreover, by a suitable change of order of $n_1$, $n_2$, $n_3$, $n_4$, we have $f=(\al_2-1)n_2-n_1=(\al_3-1)n_3-n_2=(\al_4-1)n_4-n_3=(\al_1-1)n_1-n_4$. Therefore, $\al_1$, $\al_2$, $\al_3$ and $\al_4$ are odd.
\end{remark}

\begin{example} \rm
There are almost symmetric 4-generated semigroups with type three whose pseudo-Frobenius numbers have the structure of Theorem \ref{type 3}, even though some generators are even. For instance, if $S=\langle 4,7,10,13 \rangle$, then $\PF(S)=\{3,6,9\}$. Moreover, also this semigroup is in the case nUF2, since
\begin{equation*}
\RF(3)=
\begin{pmatrix}
   -1    &    1    &    0    &    0    \\
    0    &   -1    &    1    &    0    \\
    0    &    0    &   -1    &    1    \\
    4    &    0    &    0    &   -1    \\
\end{pmatrix} \ \ \ \ {\rm and} \ \ \ \ 
\RF(6)=
\begin{pmatrix}
   -1    &    0    &    1    &    0    \\
    0    &   -1    &    0    &    1    \\
    4    &    0    &   -1    &    0    \\
    3    &    1    &    0    &   -1    \\
\end{pmatrix}
\end{equation*}

Note that $\al_2=\al_3=\al_4=2$ is even in this example.
\end{example}

By Theorem \ref{type 3} in every row of $\RF(f)$ there is exactly one positive entry. Therefore, we immediately get the following corollary by \cite[Section 5.5]{E} or \cite[Lemma 5.4]{HW}.

\begin{corollary} \label{HW}
Let $S=\langle n_1, \dots, n_4 \rangle$ be almost symmetric with type three and assume that $n_i$ is odd for every $i=1, \dots, 4$. Then
\begin{align*}
&n_1=(\al_2-1)(\al_3-1)\al_4+\al_2,  &&n_2=(\al_3-1)(\al_4-1)\al_1+\al_3, \\
&n_3=(\al_4-1)(\al_1-1)\al_2+\al_4,  &&n_4=(\al_1-1)(\al_2-1)\al_3+\al_1,
\end{align*}
where $\al_1, \ldots , \al_4$ are odd and the defining ideal of $S$ is $I_S=(x_1^{\al_1}-x_2^{\al_2-1}x_4, x_2^{\al_2}-x_3^{\al_3-1}x_1, x_3^{\al_3}-x_4^{\al_4-1}x_2$, $x_4^{\al_4}-x_1^{\al_1-1}x_3, x_1^{\al_1-1}x_2-x_3^{\al_3-1}x_4, x_1 x_4^{\al_4-1}-x_2^{\al_2-1}x_3).$
Moreover, setting $A=k[x_1,x_2,x_3,x_4]$, the minimal free resolution of $A/I_S$ is
\begin{equation*}
0 \longrightarrow A^3 \xrightarrow{\varphi_3} A^8 \xrightarrow{\varphi_2} A^6 \xrightarrow{\varphi_1} A \longrightarrow 0
\end{equation*}
where $\varphi_1$ is the obvious one and
\begin{equation*}
\varphi_2=
\begingroup 
\setlength\arraycolsep{4pt}
\begin{pmatrix}
x_3^{\al_3-1}& x_2 & 0 & 0 & 0 & 0 & x_4^{\al_4-1} & x_3    \\
x_1^{\al_1-1}& x_4 & 0 & 0 &x_4^{\al_4-1} & x_3 & 0 & 0     \\
0 & 0 & x_1^{\al_1-1} & x_4 & x_2^{\al_2-1 }& x_1 & 0 & 0   \\
0 & 0 & x_3^{\al_3-1} & x_2 & 0 & 0 & x_2^{\al_2-1} & x_1   \\
-x_2^{\al_2-1} & -x_1 & x_4^{\al_4-1} & x_3 & 0 & 0 & 0 & 0 \\
0 & 0 & 0 & 0 & -x_3^{\al_3-1} & -x_2 & x_1^{\al_1-1} & x_4 \\
\end{pmatrix},
\endgroup
\end{equation*}

\begin{equation*}
^{t}\varphi_3=
\begingroup 
\setlength\arraycolsep{4pt}
\begin{pmatrix}
0 & x_3 & 0 & x_1 & 0 & -x_4 & 0 & -x_2 \\
x_4^{\al_4-1} & 0 & x_2^{\al_2-1} & 0 & -x_1^{\al_1-1} & 0 & -x_3^{\al_3-1} & 0 \\
-x_3 & -x_4^{\al_4-1} & -x_1 & -x_2^{\al_2-1} & x_4 & x_1^{\al_1-1} & x_2 & x_3^{\al_3-1} \\
\end{pmatrix}.
\endgroup
\end{equation*}
\end{corollary}

\vspace{0.5em}

\begin{example} \rm
Putting $(\al_1,\al_2, \al_3,\al_4) = (5,3,3,3)$ in Corollary \ref{HW}, we get the semigroup $S=\langle 15,23,27,29 \rangle$. The set of its pseudo-Frobenius numbers is $\PF(S)=\{31,62,93\}$ and, then, $S$ is almost symmetric with type three. According to Theorem \ref{type 3} we have
\begin{equation*}
\RF(31)=
\begin{pmatrix}
   -1    &    2    &    0    &    0    \\
    0    &   -1    &    2    &    0    \\
    0    &    0    &   -1    &    2    \\
    4    &    0    &    0    &   -1    \\
\end{pmatrix} \ \ \ \ {\rm and} \ \ \ \ 
\RF(62)=
\begin{pmatrix}
   -1    &    1    &    2    &    0    \\
    0    &   -1    &    1    &    2    \\
    4    &    0    &   -1    &    1    \\
    3    &    2    &    0    &   -1    \\
\end{pmatrix}.
\end{equation*}
Obviously, this is the example with \lq\lq smallest" generators. 
\end{example}

\begin{theorem} \label{construction}
Assume that $\al_1$, $\al_2$, $\al_3$, $\al_4$ are odd integers greater than $1$ and let $n_1$, $n_2$, $n_3$, $n_4$ be as in Corollary \ref{HW}. If $\gcd(n_1, n_2, n_3, n_4)=1$, then $S=\langle n_1, n_2, n_3, n_4 \rangle$ is an almost symmetric semigroup generated by odd integers and has type three. Moreover, all the $4$-generated almost symmetric semigroups with type $3$ and odd generators arise in this way.
\end{theorem}

\begin{proof}
Bearing in mind Corollary \ref{HW}, it is easy to see that the ideal $I_S$ contains
\begin{align*}
J=(&x_1^{\al_1}-x_2^{\al_2-1}x_4, x_2^{\al_2}-x_3^{\al_3-1}x_1, x_3^{\al_3}-x_4^{\al_4-1}x_2, x_4^{\al_4}-x_1^{\al_1-1}x_3, \\ &x_1^{\al_1-1}x_2-x_3^{\al_3-1}x_4, x_1 x_4^{\al_4-1}-x_2^{\al_2-1}x_3).
\end{align*}
Let $A=k[x_1,x_2,x_3,x_4]$. Since $\gcd(n_1, n_2, n_3, n_4)=1$, the $k$-vector space $A/(I_S+(x_1))$ has dimension $\dim_k A/(I_S+(x_1))=\dim_k k[S]/(t^{n_1})=n_1$. Moreover,
\begin{equation} \label{ideal}
J+(x_1)=(x_2^{\al_2-1}x_4, x_2^{\al_2}, x_3^{\al_3}-x_4^{\al_4-1}x_2, x_4^{\al_4}, x_3^{\al_3-1}x_4, x_2^{\al_2-1}x_3)
\end{equation}
and it is not difficult to see that $\dim_k A/(J+(x_1))=n_1$. It follows that $A/(I_S+(x_1))=A/(J+(x_1))$ and this implies $I_S=J$, see the last part of the proof of \cite[Theorem 4.4]{HW}. 

We note that the socle of $A/(I_S+(x_1))$, defined as
\begin{equation*}
{\rm Soc} (A/(I_S+(x_1)))=\{y \in A/(I_S+(x_1)) \mid yx_2=yx_3=yx_4=0\},
\end{equation*}
is generated by $y_1=x_2^{\al_2-1}$, $y_2=x_2^{\al_2-2}x_3^{\al_3-1}$, $y_3=x_2^{\al_2-2}x_3^{\al_3-2}x_4^{\al_4-1}=x_2^{\al_2-3}x_3^{2\al_3-2}$. Therefore, the type of the ring $A/(I_S+(x_1))$ and, then, of $S$ is three. Moreover, the pseudo-Frobenius numbers of $S$ are $f_i=\deg y_i-n_1$ for $i=1,2,3$ and
\begin{gather*}
\F(S)=f_3=(\al_2-3)n_2+(2\al_3-2)n_3-n_1= \\
=(\al_2-1)n_2-n_1+(\al_2-2)n_2+(\al_3-1)n_3-n_1=f_1+f_2,
\end{gather*}
since $x_2^{\al_2}-x_3^{\al_3-1}x_1 \in I_S$.
This implies that $S$ is almost symmetric with type three and, of course, it has embedding dimension four. 
The last statement of the theorem follows from Corollary \ref{HW}.
\end{proof}

\begin{example} \rm
Let $n$ be a positive integer and set $\al_2=\al_3=\al_4=3$, $\al_1=3+2^n$. By Theorem \ref{construction} the semigroup
\begin{equation*}
S_n=\langle 15, \, 15+2^{n+2}, \, 15+2^{n+2}+2^{n+1}, \, 15+2^{n+2}+2^{n+1}+2^n \rangle. 
\end{equation*}
is an almost symmetric semigroup with type three generated by four odd minimal generators. Moreover, using Theorem \ref{type 3} it is easy to see that
$\PF(S_n)=\{15+2^{n+3}, \, 2(15+2^{n+3}), \, 3(15+2^{n+3})\}$.
\end{example}

\begin{remark} \rm The table below shows the number of almost symmetric semigroups that are minimally generated by $3$ or $4$ odd generators less than $100$, $150$ and $200$ respectively. These numbers are obtained using the GAP system \cite{GAP} and, in particular, the NumericalSgps package \cite{DGM}. 
In the table $e$ denotes the embedding dimension of $S$, $t$ denotes its type and c.i. stands for complete intersection.

\vspace{0.5em}

\begin{center}
\begin{tabular}{|c|c|c|c|} 
\hline
\ \ AS semigroups with odd gen. \ \ & \ \ Gen. $\leq 100$ \ \ & \ \ Gen. $\leq 150$ \ \ & \ \ Gen. $\leq 200$ \ \ \\ \hline
$e=3$ \ and \ $t=1$ & 2302 & 7978  & 18751 \\ \hline
$e=3$ \ and \ $t=2$ & 139  & 290   & 503   \\ \hline
$e=4$, \ $t=1$ not c.i. & 1927 & 7129 & 17524  \\ \hline
$e=4$ and c.i. & 596 & 4583 & 16895  \\ \hline
$e=4$ \ and \ $t=2$ & 595  & 1647  & 3481  \\ \hline
$e=4$ \ and \ $t=3$ & 9    & 24    & 45    \\ \hline
\end{tabular}
\end{center}

\vspace{0.5em}

It is not known if there is a bound for the type of an almost symmetric numerical semigroup with more than $4$ generators. However, some computations suggest that in the case of $5$ generators the type is at most $5$. In the following table we show the number of almost symmetric semigroups generated by 5 odd integers.

\vspace{0.5em}

\begin{center}
\begin{tabular}{|c|c|c|c|} 
\hline
\ \ AS semigroups with odd gen. \ \ & \ \ Gen. $\leq 100$ \ \ & \ \ Gen. $\leq 150$ \ \ & \ \ Gen. $\leq 200$ \ \ \\ \hline
$e=5$, \ $t=1$ not c.i. & 3451 & 19060 & 60711  \\ \hline
$e=5$ and c.i. & 0 & 135 & 1199  \\ \hline
$e=5$ \ and \ $t=2$ & 1254 & 4592  & 11489  \\ \hline
$e=5$ \ and \ $t=3$ & 988  & 3582  &  8306  \\ \hline
$e=5$ \ and \ $t=4$ & 359  & 970   &  1881  \\ \hline
$e=5$ \ and \ $t=5$ & 2    & 4     &     6  \\ \hline
\end{tabular}
\end{center}
\end{remark}

\vspace{0.5em}

\begin{example} \rm
The type of an almost symmetric semigroup may be greater than its embedding dimension, even though all the generators are odd. For instance, the 7-generated semigroup $S=\langle 29,33,61, 65, 73, 81, 85 \rangle$  is almost symmetric and its type is $12$, in fact
$\PF(S)=\{ 69, 77, 89, 93, 97, 101, 105, 109, 113, 125, 133, 202 \}$.
\end{example}

{\bf Acknowledgments.} This work began when the second author was visiting the University of Catania and he would like to express his hearty thanks for the hospitality of Marco D'Anna.
The first author was supported by INdAM, more precisely he was ``titolare di una borsa per l'estero dell'Istituto Nazionale di Alta Matematica''.

%
%

\end{document}